\documentclass[12pt]{amsart}
\usepackage[utf8]{inputenc}
\usepackage{amsfonts}
\usepackage{amsmath}
\usepackage{amssymb}
\usepackage{amsthm}
\usepackage[margin=1.3in]{geometry}
\usepackage{graphicx}

\newtheorem{theorem}{Theorem}
\newtheorem*{conjecture*}{Conjecture}

\newtheorem{lemma}{Lemma}

\theoremstyle{remark}

\newtheorem*{remark*}{Remark}

\newcommand{\vs}{\vspace{3mm}}
\newcommand{\vsss}{\vspace{6mm}}

\newcommand{\hs}{\hspace{1mm}}

\newcommand{\Z}{\mathbb{Z}}

\newcommand{\N}{\mathbb{N}}
\newcommand{\C}{\mathbb{C}}

\newcommand{\ep}{\epsilon}
\newcommand{\lam}{\lambda}
\newcommand{\sub}{\subseteq}
\newcommand{\al}{\alpha}

\DeclareMathOperator{\Probb}{Prob}
\DeclareMathOperator{\supp}{supp}

\title{The Maximum Number of Three Term Arithmetic Progressions, and Triangles in Cayley Graphs}
\author{Zachary Chase}
\address{Department of Mathematics, California Institute of Technology, Pasadena, CA, 91125}
\email{zchase@caltech.edu}
\date{September 11th, 2018}

\begin{document}

\begin{abstract}
Let $G$ be a finite Abelian group. For a subset $S \subseteq G$, let $T_3(S)$ denote the number of length three arithemtic progressions in $S$ and $\Probb[S] = \frac{1}{|S|^2}\sum_{x,y \in S} 1_S(x+y)$. For any $q \ge 1$ and $\al \in [0,1]$, and any $S \subseteq G$ with $|S| = \frac{|G|}{q+\al}$, we show $\frac{T_3(S)}{|S|^2}$ and $\Probb[S]$ are bounded above by $\max\left(\frac{q^2-\al q+\al^2}{q^2},\frac{q^2+2\al q+4\al^2-6\al+3}{(q+1)^2},\gamma_0\right)$, where $\gamma_0 < 1$ is an absolute constant. As a consequence, we verify a graph theoretic conjecture of Gan, Loh, and Sudakov for Cayley graphs.
\end{abstract}

\maketitle

\section{Introduction}

The study of arithmetic progressions in subsets of integers and general Abelian groups is a central topic in additive combinatorics and has led to the development of many fascinating areas of mathematics. A famous result on three term arithmetic progressions (3APs) is Roth's theorem, which, in its finitary form, says that for each $\lam > 0$, for $N$ large, any subset $S \sub \{1,\dots, N\}$ of size $|S| \ge \lam N$ contains a 3AP.

\vs

Once Roth's theorem ensures that all subsets of a given size have a 3AP, one can generate many 3APs. For example, Varnavides [4] proved that for each $\lam > 0$, there is some $c > 0$ so that for all large $N$, every subset $S \sub \{1,\dots,N\}$ with $|S| \ge \lam N$ contains at least $cN^2$ 3APs. A natural question is then how many 3APs a subset of $\{1,\dots,N\}$ of a prescribed size can have. We look at this question in the group theoretic setting.

\vs

Fix $\lam \in (0,1)$. Let $p$ be a large prime and consider subsets $S \sub \Z_p$ of size $|S| = \lfloor \lam p \rfloor$. If $T_3(S)$ denotes the number of 3APs in $S$, namely, the number of $x,d \in \Z_p$ with $x,x+d,x+2d \in S$, then Croot [1] showed that $$\lim_{p \to \infty} \max_{\substack{S \sub \Z_p \\ |S| = \lfloor \lam p \rfloor}} \frac{T_3(S)}{|S|^2}$$ exists, and then Green and Sisask [2] proved that the limit is in fact $\frac{1}{2}$, for all $\lam$ less than some absolute constant. In $\Z_n$, for $n$ not prime, the situation is quite different, since subgroups have many 3APs relative to their size. In this paper, we nevertheless get an upper bound, useful when the size of $S$ is ``far" from dividing $n$.

\vs

\begin{theorem} There is an absolute constant $\gamma_1 < 1$ so that for any finite Abelian group $G$ of odd order, and for any $q \in \N, \al \in [0,1]$, $$\max_{\substack{S \sub G \\ |S| = \frac{|G|}{q+\al}}} \frac{T_3(S)}{|S|^2} \le \max\left(\frac{q^2-\al q+\al^2}{q^2},\frac{q^2+2\al q+4\al^2-6\al+3}{(q+1)^2},\gamma_1\right).$$
\end{theorem}

\vspace{1mm}

Related to $\frac{T_3(S)}{|S|^2} = \frac{1}{|S|^2} \sum_{x,y \in S} 1_S(\frac{x+y}{2})$ is the quantity $\frac{1}{|S|^2}\sum_{x,y \in S} 1_S(x+y)$. This quantity, which we denote $\Probb[S]$, arises in the expression for the number of triangles in a Cayley graph with generating set $S$. Precisely, let $G$ be an additive group of size $n$ and $S \sub G$ a symmetric set not containing $0$. Connect $x,y \in G$ iff $x-y \in S$. We obtain an undirected graph on $G$ with no self loops. The number of triangles in our graph is $$\frac{1}{6}\sum_{a,b,c \in G} 1_S(a-b)1_S(b-c)1_S(a-c).$$ Let $x = a-b$ and $y = b-c$. Then ranging over $c,b,a$ is equivalent to ranging over $c,y,x$ and thus $$|T| = \frac{1}{6}\sum_{x,y,c} 1_S(x)1_S(y)1_S(x+y) = \frac{1}{6}n\sum_{x,y \in S} 1_S(x+y) = \frac{1}{6}n|S|^2\Probb[S].$$

Quite recently, Gan, Loh, and Sudakov [3] resolved a conjecture of Engbers and Galvin regarding the maximum number of independent sets of size $3$ that a graph with a given minimum degree and fixed size can have. Phrased in complementary graphs, they showed that given a maximum degree $d$ and a positive integer $n \le 2d+2$, the maximum number of triangles that a graph on $n$ vertices with maximum degree $d$ can have is ${d+1 \choose 3}+{n-(d+1) \choose 3}$. This immediately raised the question of what the maximum is for $n > 2d+2$. They conjectured the following.

\vs

\begin{conjecture*}[Gan-Loh-Sudakov] Fix $d \ge 2$. For any positive integer $n$, if we write $n = q(d+1)+r$ for $0 \le r \le d$, then the maximum number of triangles that a graph on $n$ vertices with maximum degree $d$ can have is $q{d+1 \choose 3}+{r \choose 3}$.
\end{conjecture*}

\vs

For each $d,n$, an example of a graph achieving $q{d+1 \choose 3}+{r \choose 3}$ is simply a disjoint union of $K_{d+1}$'s and a $K_r$. The conjecture for a Cayley graph on an additive group $G$ with generating set $S$, $|S| = \frac{|G|}{q+\al}$, takes the form $\Probb[S] \le \frac{q+\al^3}{q+\al}$, up to smaller order terms. We verify the conjecture for Cayley graphs when $q \ge 7$.

\vs

\begin{theorem} There is an absolute constant $\gamma_0 < 1$ so that the following holds. Let $G$ be a finite Abelian group and take $q \in \N, \al \in [0,1]$. Then for any symmetric subset $S \sub G$ with $|S| = \frac{|G|}{q+\al}$, $$\frac{1}{|S|^2}\sum_{x,y \in S} 1_S(x+y) \le \max\left(\frac{q^2-\al q+\al^2}{q^2},\frac{q^2+2\al q+4\al^2-6\al+3}{(q+1)^2},\gamma_0 \right).$$ Consequently, the Gan-Loh-Sudakov conjecture holds for Cayley graphs with generating set $|S| \le \frac{n}{7}$.
\end{theorem}

\vs

We give a fourier analytic proof of Theorems 1 and 2. Here is a quick high-level overview of the argument. We express the relevant ``probability" (either $\frac{1}{|S|^2}\sum_{x,y \in S} 1_S(\frac{x+y}{2})$ or $\frac{1}{|S|^2}\sum_{x,y \in S} 1_S(x+y)$) in terms of the fourier coefficients of $1_S$. If the probability is large, then some nonzero fourier coefficient must be large. We deduce that (a dilate of) the residues of $S$ of a certain modulus concentrate near $0$. Since there won't be ``wraparound" near 0, this allows us to transfer the problem to $\Z$, which is a setting where it's easier to bound the relevant probabilities. We can show from the result in $\Z$ that we in fact must have many residues be $0$. This allows us to conclude that $S$ is very close to a subgroup. Induction and a purely combinatorial argument finish the job from there.

\vs

Here is an outline of the paper. We first set our notation for Fourier analysis on $\Z_n$. Then we give the proof of Theorems 1 and 2, modulo two Lemmas, which we prove afterwards. After, we show the calculations deducing the Gan-Loh-Sudakov conjecture from our main theorem. Finally, we prove Theorems 1 and 2 when $q=1$.

\section{Fourier Analysis on $\Z_n$}

In this section, we briefly fix our notation for fourier analysis on $\Z_n$ and obtain the fourier representation of the relevant quantities in the proofs to be given below. For a function $f: \Z_n \to \C$, define its (finite) fourier transform $\widehat{f} : \Z_n \to \C$ by $$\widehat{f}(m) := \frac{1}{n}\sum_{x \in \Z_n} f(x)e^{-2\pi i \frac{xm}{n}}.$$ The following well-known equalities are straightforward. $$\sum_{m \in \Z_n} |\widehat{f}(m)|^2 = \frac{1}{n} \sum_{x \in \Z_n} |f(x)|^2$$ $$ f(x) = \sum_{m \in \Z_n} \widehat{f}(m)e^{2\pi i \frac{xm}{n}}.$$ Let $S$ be a symmetric subset of $\Z_n$. Then, $\frac{1}{|S|^2}\sum_{x,y \in S} 1_S(x+y) = $ $$\frac{1}{|S|^2} \sum_{x,y \in \Z_n} \left[\sum_{m_1 \in \Z_n} \widehat{1_S}(m_1)e^{2\pi i \frac{xm_1}{n}}\right]\left[\sum_{m_2 \in \Z_n} \widehat{1_S}(m_2)e^{2\pi i \frac{ym_2}{n}}\right] \left[\sum_{m_3 \in \Z_n} \widehat{1_S}(m_3)e^{2\pi i \frac{(x+y)m_3}{n}}\right]$$ $$= \frac{1}{|S|^2}\sum_{m_1,m_2,m_3 \in \Z_n} \widehat{1_S}(m_1)\widehat{1_S}(m_2)\widehat{1_S}(m_3) \left[\sum_{x \in \Z_n} e^{2\pi i \frac{x(m_1+m_3)}{n}}\right] \left[\sum_{y \in \Z_n} e^{2\pi i \frac{y(m_2+m_3)}{n}}\right],$$ and using $$\sum_{x \in \Z_n} e^{2\pi i \frac{xk}{n}} = \begin{cases} n & k \equiv 0 \pmod{n} \\ 0 & k \not \equiv 0 \pmod{n} \end{cases},$$ we obtain $$\frac{1}{|S|^2}\sum_{x,y \in S} 1_S(x+y) =\frac{n^2}{|S|^2} \sum_{m \in \Z_n} \widehat{1_S}(-m)\widehat{1_S}(-m)\widehat{1_S}(m).$$ However, the symmetry of $S$ implies that $\widehat{1_S}(m) = \widehat{1_S}(-m)$ for each $m \in \Z_n$. Therefore, $$\Probb[S] = \frac{1}{|S|^2}\sum_{x,y \in S} 1_S(x+y) = \frac{n^2}{|S|^2} \sum_{m \in \Z_n} \widehat{1_S}(m)^3.$$ Similarly, for any subset $S \sub \Z_n$, $$\frac{1}{|S|^2}\sum_{x,y \in S} 1_S(\frac{x+y}{2}) = \frac{n^2}{|S|^2}\sum_{m \in \Z_n} \widehat{1_S}(m)^2\widehat{1_S}(-2m).$$

\vs

\section{Proof of Theorems 1 and 2}

We induct on $q$. We discuss the base case $q=1$ in section 6. Take some $q \ge 2$ and $\al \in [0,1]$. Let $S \sub \Z_n$ be a symmetric\footnote{In the 3AP setting, we do not assume $S$ is symmetric.} subset with $|S| = \frac{n}{q+\al}$.

\vs

Let $\gamma = \max(\frac{q^2-\al q+\al^2}{q^2},\frac{q^2+2\al q+4\al^2-6\al+3}{(q+1)^2},\gamma_0)$. Assume, for the sake of contradiction, that $\Probb[S] \ge \gamma$. Then, as explained in section 2, $$\sum_m \widehat{1_S}(m)^3 \ge \frac{d^2}{n^2}\gamma.$$ Note $\widehat{1_S}(0)^3 = \frac{d^3}{n^3}$, so, since $\widehat{1_S}(m)$ is real for each $m$\footnote{In the 3AP setting, we instead do $\gamma \frac{d^2}{n^2}-\frac{d^3}{n^3} \le \sup_{m \not = 0} |\widehat{1_S}(-2m)| \cdot [\frac{d}{n}-\frac{d^2}{n^2}]$. Then we take $m_0$ with $|\widehat{1_S}(m_0)| \ge \frac{d}{n}\mu$. Finally, we can translate $S$ so that $\widehat{1_S}(m_0)$ is real and positive.}, $$\gamma\frac{d^2}{n^2}-\frac{d^3}{n^3} \le \sum_{m \not = 0} \widehat{1_S}(m)^3 \le \left(\sup_{m \not = 0} \widehat{1_S}(m)\right) \cdot \sum_{m \not = 0} \widehat{1_S}(m)^2 = \left(\sup_{m \not = 0} \widehat{1_S}(m)\right) \cdot [\frac{d}{n}-\frac{d^2}{n^2}],$$ where we used Plancherel in the last step. Take $m_0 \not = 0$ with $$\widehat{1_S}(m_0) \ge \frac{d}{n}\frac{\gamma-\frac{d}{n}}{1-\frac{d}{n}} =: \frac{d}{n}\mu.$$

Then, $$\mu \le \frac{1}{d}\sum_{x \in S} e^{2\pi i \frac{m_0}{n}x} = \frac{1}{d}\sum_{x \in S} e^{2\pi i \frac{m_0/g}{n/g}x},$$ where $g := \gcd(m_0,n)$. Let $$A = \{x \in \Z_n : 2\pi \frac{m_0/g}{n/g}x \in [-2\pi/3,2\pi/3] \pmod{2\pi}\}$$ $$B = \Z_{n/g}\setminus A.\footnote{In the 3AP setting, we let $A = \{x \in \Z_n : 2\pi \frac{m_0/g}{n/g}x \in [-\frac{\pi}{2},\frac{\pi}{2}]\}$ and $B = \Z_{n/g} \setminus A$.}$$

Then, since $\widehat{1_S}(m_0)$ is real, $$d\mu \le \sum_{x \in S} \cos(2\pi \frac{m_0/g}{n_0/g}x) \le |A|+(d-|A|)(-\frac{1}{2}),$$ which implies $$ \frac{|A|}{d} \ge \frac{2\mu+1}{3}. \footnote{In the 3AP setting, we get $d\mu \le |A|+(d-|A|)0$ and thus $\frac{|A|}{d} \ge \mu$.} $$ For $z \in B$, $$\#\{(x,y) \in S^2 : x+y = z\} \le d$$ and for $z \in A$, $$\#\{(x,y) \in B\times A : x+y = z\} \le |B|$$ $$\#\{(x,y) \in S\times B : x+y=z\} \le |B|$$ $$\#\{(x,y) \in A \times A : x+y = z\} =: C_z. \footnote{In the 3AP setting, the sets will merely have $2z$ instead of $z$ - the same estimates thus hold.}$$ Therefore,, $$d^2\Probb[S] \le d|B|+2|A|\hs|B|+\sum_{z \in A} C_z$$ $$= d(d-|A|)+2|A|(d-|A|)+|A|^2\Probb[A].$$

\noindent So, we must have $$\Probb[A] \ge \frac{\gamma+2\frac{|A|^2}{d^2}-\frac{|A|}{d}-1}{\frac{|A|^2}{d^2}}.$$ If we let $f(x) = \frac{\gamma+2x^2-x-1}{x^2}$, then $f'(x) = -2\gamma x^{-3}+x^{-2}+2x^{-3}$ is positive for $x > 0$. We've shown $\frac{|A|}{d} \ge \frac{2\mu+1}{3} =: v \footnote{In the 3AP setting, we have $\nu := \mu$.}$, so we get that $$\Probb[A] \ge \frac{\gamma+2v^2-v-1}{v^2} =: \beta.$$

We now argue that the weight at $0$ must be large. For each $i \in [-\frac{1}{3}\frac{n}{g},\frac{1}{3}\frac{n}{g}]$, let $S_i = \{x \in S : x \equiv i \pmod{n/g}\}$. Let $a_i = |S_i|$. Note that for each $i,j \in [-\frac{1}{3}\frac{n}{g},\frac{1}{3}\frac{n}{g}]$ such that $i+j \in [-\frac{1}{3}\frac{n}{g},\frac{1}{3}\frac{n}{g}]$, $$\#\{(x_i,y_j,z_{i+j}) \in S_i\times S_j \times S_{i+j} : x_i+y_j = z_{i+j}\} \le \min(|S_i|\hs |S_j|, |S_i| \hs |S_{i+j}|, |S_j| \hs |S_{i+j}|). \footnote{In the 3AP setting, we'll be looking at $[-\frac{1}{4}\frac{n}{g},\frac{1}{4}\frac{n}{g}]$ instead. Also, we'll have $2z_{\frac{i+j}{2}} \in S_{\frac{i+j}{2}}$ instead of $z_{i+j} \in S_{i+j}$, and $|S_{\frac{i+j}{2}}|$ instead of $|S_{i+j}|$. This alters Lemma 1 not too significantly.} $$ The uniqueness of $0$ is that $0+0 = 0$, so that $\#\{(x_0,y_0,z_0) \in S_0^3 : x_0+y_0 = z_0\}$ cannot be upper bounded by potentially smaller terms $|S_i|, i \not = 0$. Note that the sets whose size we just bounded account for all the terms in the computation of $\Probb[A]$, since, by our choice of $A$, there is no ``wraparound".\footnote{In the 3AP setting, the lack of wraparound for $x,y \in [-\frac{1}{4}\frac{n}{g},\frac{1}{4}\frac{n}{g}] \pmod{n/g}$ follows from the fact that either $x+y$ is even and then of course $\frac{x+y}{2} \in [-\frac{1}{4}\frac{n}{g},\frac{1}{4}\frac{n}{g}]$, or it's odd and then $\frac{x+y}{2} = (x+y)\frac{n+1}{2} = \frac{x+y-1}{2}+\frac{g-1}{2}\frac{n}{g}+\frac{\frac{n}{g}+1}{2} = \frac{x+y-1}{2}+\frac{\frac{n}{g}+1}{2} \pmod{n/g}$; since $\frac{x+y-1}{2} \in [-\frac{1}{4}\frac{n}{g},\frac{1}{4}\frac{n}{g}]$ we therefore see that $\frac{x+y}{2} \not\in [-\frac{1}{4}\frac{n}{g},\frac{1}{4}\frac{n}{g}] \pmod{n/g}$.}

Take $\gamma_0$ so that $\beta > \frac{9}{10}$ (for any $q,\al$). $\gamma_0 = .949$ works\footnote{In the 3AP setting, we get a larger value for $\gamma_1$, but of course, a value less than $1$.}. Then Lemma 1 applies and we obtain, $$\frac{|S_0|}{|A|} \ge \Probb[A] \ge \beta.$$ It should be noted that we already get a contradiction if $g \le \beta\nu d$ since we clearly must have $|S_0| \le g$. In any event, we argue that this large a weight at $0$ forces $S$ to be close enough to the subgroup $\{0,\frac{n}{g},\frac{2n}{g},\dots,\frac{(g-1)n}{g}\}$ for us to get a direct upper bound on $\Probb[S]$. For ease, let $$D = \{x \in S : x \equiv 0 \pmod{n/g}\}$$  $$E = S\setminus D.$$

Then, $$\Probb[S] = \frac{1}{d^2}\sum_{x,y \in S} 1_S(x+y)$$ $$= \frac{|D|^2}{d^2}\frac{1}{|D|^2}\sum_{x,y \in D} 1_S(x+y) + \frac{2}{d^2}\sum_{x \in D, y \in E} 1_S(x+y) + \frac{1}{d^2}\sum_{x,y \in E} 1_S(x+y).$$ Using that $D$ is contained in a subgroup disjoint from $E$, we have the following (in)equalities $$\sum_{x,y \in D} 1_S(x+y) = \sum_{x,y \in D} 1_D(x+y)$$ $$\sum_{x \in D, y\in E} 1_S(x+y) = \sum_{x \in D, y\in E} 1_E(x+y) = \sum_{y \in E} \sum_{x \in D} 1_{-y+E}(x) \le \sum_{y \in E} |E|$$ $$\sum_{x,y \in E} 1_S(x+y) \le |E|^2. \footnote{In the 3AP setting, we replace $x+y$ with $\frac{x+y}{2}$. If $x,y \in D$, then $\frac{x+y}{2} \in D$. And if $x \in D, y \in E$, then $x+y$ can't be in $2^{-1}D = D$. The three analogous (in)equalities thus hold.} $$ Hence, $$\Probb[S] \le \frac{|D|^2}{d^2}\Probb[D]+\frac{3}{d^2}|E|^2.$$

Using a cheaper ``approximation" argument, similar to the one used previously, that doesn't capitalize on the fact that $D$ is contained in a subgroup disjoint from $E$ will yield an upper bound for $\Probb[S]$ larger than $1$.

Note $\frac{|D|}{d} = \frac{|D|}{|A|}\frac{|A|}{d} \ge \beta\nu$. Let $\eta = \frac{|D|}{d}, k = \frac{n}{g} \in \N, q' = \lfloor \frac{g}{|D|} \rfloor$, and $\al' = \frac{g}{|D|}-q'$. Then by induction and the obvious observation that $\Probb[D]$ is independent of whether the ambient group is $\Z_n$ or $\{0,\frac{n}{g},\dots,(g-1)\frac{n}{g}\}$,  $$\Probb[D] \le \max\left(\frac{(q')^2-\al' q'+(\al')^2}{(q')^2},\frac{(q')^2+2\al' q'+4(\al')^2-6\al'+3}{(q'+1)^2},\gamma_0\right);$$ hence, $$\Probb[S] \le \eta^2 \max\left(\frac{(q')^2-\al' q'+(\al')^2}{(q')^2},\frac{(q')^2+2\al' q'+4(\al')^2-6\al'+3}{(q'+1)^2},\gamma_0\right) +3(1-\eta)^2.$$ Note that the induction is justified, as $q' = \lfloor \frac{g}{|D|} \rfloor \le \frac{g}{|D|} < q$, since $\frac{g}{|D|} \le \frac{n/2}{\beta v d} \le \frac{n/2}{\frac{3}{4}d} = \frac{2}{3}(q+\al)$, where we used that $\beta v \ge \frac{3}{4}$, which holds for $q \ge 2$. We finish by appealing to Lemma 2, which indeed applies when $\beta\nu \ge \frac{3}{4}$.

\vs

The above proof readily extends to an arbitrary finite Abelian group. Fix $r \ge 1$ and positive integers $n_1,\dots, n_r$. Let $n = n_1\dots n_r$ and $S$ be a subset of $\Z_{n_1}\times \dots \times \Z_{n_r}$ of size $|S| = \frac{n}{q+\al}$. Since $\widehat{1_S}(0,\dots,0) = \frac{|S|}{n}$ and Plancherel holds, there is some $(m_1,\dots,m_r) \not = (0,\dots,0)$ with $$\frac{d}{n}\mu := \frac{d}{n}\frac{\gamma-\frac{d}{n}}{1-\frac{d}{n}} \le \widehat{1_S}(m_1,\dots,m_r) = \frac{1}{n}\sum_{(x_1,\dots,x_r) \in S} e^{2\pi i (\frac{m_1x_1}{n_1}+\dots+\frac{m_rx_r}{n_r})}.$$ Analogous to before, letting $A = \{(x_1,\dots,x_r) \in S : 2\pi(\frac{m_1x_1}{n_1}+\dots+\frac{m_rx_r}{n_r}) \in [\frac{-2\pi}{3},\frac{2\pi}{3}] \pmod{2\pi}\}$, we must have $\frac{|A|}{d} \ge \frac{2\mu+1}{3}$. Let $S_j = \{(x_1,\dots,x_r) \in S : e^{2\pi i (\frac{m_1x_1}{n_1}+\dots+\frac{m_rx_r}{n_r})} = e^{2\pi i \frac{j}{n}}\}$. Then, as before, we must have $\frac{|S_0|}{|A|} \ge \beta$. But $S_0$ is a subgroup of $\Z_{n_1} \times \dots \times \Z_{n_r}$, so the same inductive argument finishes the job. \qed

\vsss

\section{Proof of Lemmas}

\begin{lemma} Fix $d \ge 1$ and $\ep \in [0,\frac{1}{10})$. Let $\{a_j\}_{j \in \Z}$ be a collection of non-negative integers such that $\sum_{i \in \Z} a_i = d$ and $a_j = a_{-j}$ for each $j \in \Z$. Then if $$\sum_{i,j} \min(a_ia_j,a_ia_{i+j},a_ja_{i+j}) \ge (1-\ep)d^2,$$ we must have that $$a_0 \ge (1-\ep)d.$$
\end{lemma}

\begin{proof}
Define $\supp(a_j) := \supp((a_j)_{j \in \Z}) := \#\{n \ge 1 : a_n \not = 0\}$. We induct on $\supp(a_j)$, with base case $\supp(a_j)=0$ obvious. Let $(a_j)_{j \in \Z}$ have $\supp(a_j) =: N+1$. Let $n+1$ be the largest index $j$ for which $a_j \not = 0$. First assume that $a_{n+1} \le \frac{1}{10}d$. Define $(b_j)_{j \in \Z}$ via $b_j = a_j$ if $|j| \le n$ and $b_j = 0$ if $|j| \ge n+1$. Then $b_j = b_{-j}$ for $j \in \Z$, $\supp(b_j) \le N$, and $\sum_{j \in \Z} b_j = d-2a_{n+1}$. Note that $$A_{n+1} := \sum_{i,j} \min(a_ia_j,a_ia_{i+j},a_ja_{i+j})$$ $$ \le \sum_{i,j} \min(b_ib_j,b_ib_{i+j},b_jb_{i+j}) + 2\sum_{k=1}^n a_ka_{n+1} + 4\sum_{-n \le k \le -1} a_{n+1}a_k + 2a_{n+1}^2+4a_{n+1}^2$$ $$=: A_n + 6a_{n+1}(\frac{d-a_0-2a_{n+1}}{2}) + 6a_{n+1}^2.$$ Here we counted the number of ways $n+1$ or $-(n+1)$ can occur as $i+j$ for $i,j \not = 0$, then the number of ways $n+1$ or $-(n+1)$ can occur as $i$ or $j$ with no $0$ as the other coordinate, and then accounted for the terms $(i,j) = (n+1,-(n+1)),(-(n+1),n+1),$ $(n+1,0),(-(n+1),0),(0,n+1)$, and $(0,-(n+1))$. If $A_{n+1} \ge (1-\ep)d^2$, then $$(*) \hspace{10mm} A_n \ge (1-\ep)d^2-3a_{n+1}(d-a_0).$$

We first show $3a_0 \ge (1+2\ep)d$. Bounding $a_0 \ge 0$ in (*) gives $$A_n \ge \frac{(1-\ep)d^2-3a_{n+1}d}{(d-2a_{n+1})^2}(d-2a_{n+1})^2.$$ To use the claim applied to $(b_j)_{j \in \Z}$ and total weight $d-2a_{n+1}$, we must check that $$1-\frac{(1-\ep)d^2-3a_{n+1}d}{(d-2a_{n+1})^2} < \frac{1}{10}.$$ It suffices to show $$1-\frac{(1-\ep)d^2-3a_{n+1}d}{(d-2a_{n+1})^2} < \ep.$$ Rearranging gives $$a_{n+1} < \frac{1-4\ep}{4(1-\ep)}d,$$ which is true for $\ep < 1/10$ and $a_{n+1} < \frac{d}{10}$. Hence, by induction, $$3a_0 \ge 3\left[\frac{(1-\ep)d^2-3a_{n+1}d}{(d-2a_{n+1})^2}\right](d-2a_{n+1}) = 3\frac{(1-\ep)d^2-3a_{n+1}d}{(d-2a_{n+1})}.$$ This is larger than $(1+2\ep)d$ iff $$a_{n+1} < \frac{2-5\ep}{7-4\ep}d.$$ This is true for $\ep < 1/10$ and $a_{n+1} < d/10$.

Now, let $\alpha$ be such that $$(1-\ep)d^2-3a_{n+1}(d-2a_{n+1}-a_0)-6a_{n+1}^2 = (1-\al)(d-2a_{n+1})^2.$$ Then, assuming $\al < \frac{1}{10}$, we can use induction to get that $$a_0 \ge (1-\al)(d-2a_{n+1}).$$ So to finish the induction, it suffices to show that $$(1-\al)(d-2a_{n+1}) \ge (1-\ep)d,$$ which is equivalent to $$\frac{(1-\ep)d^2-3a_{n+1}(d-a_0)}{d-2a_{n+1}} \ge (1-\ep)d,$$ which, after simplifying, is equivalent to $$3a_0 > (1+2\ep)d,$$ which we have proven. Therefore, all we need to do is prove $\al < \frac{1}{10}$. It suffices to show $\al < \ep$. But, as we've just noted, $(1-\al)(d-2a_{n+1}) \ge (1-\ep)d$, so $\al \le 1-\frac{(1-\ep)d}{d-2a_{n+1}} \le 1-\frac{(1-\ep)d}{d} = \ep$, as desired.

\vs

We finish by arguing that we in fact must have $a_{n+1} < \frac{d}{10}$ for $\ep < \frac{1}{10}$. First note $$\sum_{i,j} a_ia_j - \sum_{i,j} \min(a_ia_j,a_ia_{i+j},a_ja_{i+j}) \ge 4\sum_{1 \le k \le n} a_ka_{n+1}+2a_{n+1}^2.$$ Therefore, we have that $$d^2 \ge (1-\ep)d^2 + 4a_{n+1}(\frac{d-a_0-2a_{n+1}}{2})+2a_{n+1}^2$$ and hence, $$2a_{n+1}^2-2a_{n+1}(d-a_0)+\ep d^2 \ge 0.$$ As one can verify, the proof given above (for $a_{n+1} < \frac{d}{10}$) works regardless of what $a_{n+1}$ is, if $a_0 > (\frac{1+2\ep}{3})d$. Therefore, we may assume $a_0 \le (\frac{1+2\ep}{3})d$ and get that we must have $$2a_{n+1}^2 - 2a_{n+1}(\frac{2-2\ep}{3})d+\ep d^2 \ge 0.$$ So, $\frac{a_{n+1}}{d} < \frac{\frac{2-2\ep}{3}-\sqrt{(\frac{2-2\ep}{3})^2-2\ep}}{2}$ or $\frac{a_{n+1}}{d} > \frac{\frac{2-2\ep}{3}+\sqrt{(\frac{2-2\ep}{3})^2-2\ep}}{2}$. However, the first expression in $\ep$ is less than $\frac{1}{10}$ for $\ep < \frac{1}{10}$, and the second expression is greater than $\frac{1}{2}$ for $\ep < \frac{1}{10}$. Since we clearly can't have $a_{n+1} > \frac{d}{2}$, we're done.
\end{proof}

\vs

\begin{remark*} It should be noted that the largest we can possibly take $\ep$ in the statement of Lemma 1 is $\ep = \frac{2}{9}$. Consider, for example, $a_0,a_{-1},a_1 = \frac{d}{3}$. Extending Lemma 1 from $\ep < \frac{1}{10}$ to $\ep < \frac{2}{9}$ will just slightly lower the value of $\gamma_0$, and will not allow one to get all the way down to $q \le 3$. 
\end{remark*}

\begin{remark*} In the 3AP setting we may not necessarily have that $a_j = a_{-j}$ for each $j \in \Z$. However, a suitable adjustment of the given proof shows that, for $\ep$ small enough, $\sum_{i,j} \min(a_ia_j,a_ia_{\frac{i+j}{2}},a_ja_{\frac{i+j}{2}}) \ge (1-\ep)d^2$ implies $a_j \ge (1-\ep)d$ for some $j$. We can then just translate $S$ to assume $j=0$.
\end{remark*}

\vs

\begin{lemma} For $q\in \N, \al \in [0,1]$, define $$F(q,\al) = \max\left(\frac{q^2-\al q+\al^2}{q^2},\frac{q^2+2\al q+4\al^2-6\al+3}{(q+1)^2},\gamma_0 \right).$$ For any $q\ge 2, \al \in [0,1], 1 \le k \le q, \eta \in (\frac{3}{4},1]$, if we let $q' = \lfloor \frac{q+\al}{k\eta} \rfloor$ and $\al' = \frac{q+\al}{k\eta} - q'$, then $$\eta^2F(q',\al')+3(1-\eta)^2 < F(q,\al).$$
\end{lemma}

\begin{proof} Fix any $q,k,q' \ge 1$ and $\al \in [0,1]$. Substitute $\eta = \frac{q+\al}{(q'+\al')k}$ and let $$f(\al') := \frac{(q+\al)^2}{k^2}\frac{1}{(q'+\al')^2}F(q',\al')+3(1-\frac{q+\al}{(q'+\al')k})^2.$$ We show that $f(\al')$ attains its maximum at (one of) the extreme values of $\al'$. Define $$f_1(\al') := \frac{(q+\al)^2}{k^2}\frac{1}{(q'+\al')^2}\frac{(q')^2-\al' q'+(\al')^2}{(q')^2}+3(1-\frac{q+\al}{(q'+\al')k})^2$$ $$f_2(\al') := \frac{(q+\al)^2}{k^2}\frac{1}{(q'+\al')^2}\frac{(q')^2+2\al' q'+4(\al')^2-6\al'+3}{(q+1)^2}+3(1-\frac{q+\al}{(q'+\al')k})^2.$$ A straightforward computation shows $$f_1'(\al') = \frac{q+\al}{k^2}\frac{1}{(q'+\al')^3} \cdot $$ $$\bigg[ (2\al'-q')(\al'+q')(q+\al)-2((\al')^2-2q'\al'+(q')^2)(q+\al)+6(k(\al'+q')-(q+\al))\bigg]$$ $$f_2'(\al') = \frac{q+\al}{k^2}\frac{1}{(q'+\al')^3} \cdot$$ $$ \bigg[(\al'+q')(8\al'+2(q'-3))(q+\al)-2(4(\al')^2+2(q'-3)\al'+(q')^2+3)(q+\al)+6(k(\al'+q')-(q+\al))\bigg]$$

In each $f_j'(\al')$, in the brackets, the quadratic term in $\al'$ vanishes. Therefore, in the brackers is a term linear in $\al'$. In $f_1'(\al')$ the coefficient of $\al'$ is $q'(q+\al)+4q'(q+\al)+6k$, which is positive. Similarly, the coefficient of $\al'$ in $f_2'(\al')$ is $8q'(q+\al)+2(q'-3)(q+\al)-4(q'-3)(q+\al)+6k = (6q'+6)(q+\al)+6k$, which is positive. Hence, $f_1(\al'),f_2(\al')$ attain their maximum values only at the extreme values of $\al'$. Since $f(\al') = \max(f_1'(\al'),f_2'(\al'))$\footnote{Clearly $\eta^2 \gamma_0+3(1-\eta)^2 \le \gamma_0$ for $\eta \in (\frac{3}{4},1)$, since $\gamma_0 > \frac{3}{7}$. So, we assume $F(q',\al') \not = \gamma_0$.}, we see that $f(\al')$ attains its maximum at (one of) the extreme values of $\al'$.

\vs

Suppose $\frac{q+\al}{(q'+\al')k} < 1$ for some $\al' \in (0,1)$. Then $\frac{q+\al}{(q'+1)k} < 1$. Note $\al' = 1 \implies F(q',\al') = 1$, and $\eta^2+3(1-\eta)^2$ is increasing for $\eta > \frac{3}{4}$. Since $\eta > \frac{3}{4}$ and since $\eta < 1$, we take $\eta = \frac{q+\al}{q+1}$ (since $q'k \in \N$). We obtain $\frac{q^2+2\al q+4\al^2-6\al+3}{(q+1)^2}$, which, of course, is at most $F(q,\al)$.

\vs

If $\frac{q+\al}{q'k} < 1$, then we take $\al' = 0$ and argue as above. Otherwise, the extreme value of $\al'$ is the one making $\eta = 1$, namely $\al'_{crit} = \frac{q+\al}{k}-q'$. At $\eta = 1$, our desired inequality becomes $F(q',\al'_{crit}) \le F(q,\al)$. Since $\al'_{crit} \in [0,1]$ and $q' \in \N$, we have $q' = \lfloor \frac{q+\al}{k} \rfloor, \al'_{crit} = \{\frac{q+\al}{k}\}$, the fractional part. Therefore, it just suffices to show, generally, that $$q,k \ge 1, \al \in [0,1] \implies F(\lfloor \frac{q+\al}{k} \rfloor, \{\frac{q+\al}{k}\}) \le F(q,\al).$$

\vs

Clearly, the inequality holds if $F(\lfloor \frac{q+\al}{k} \rfloor, \{\frac{q+\al}{k}\}) = \gamma_0$. If $q=2$, then either $k=1$ and the inequality is an equality, or $k=2$ and $F(\lfloor \frac{q+\al}{k} \rfloor, \{\frac{q+\al}{k}\}) = F(1, \frac{\al}{2}) = 1-\frac{\al}{2}+\frac{\al^2}{4}$, while $F(q,\al) \ge \frac{4-2\al+\al^2}{4} = 1-\frac{\al}{2}+\frac{\al^2}{4}$. So, assume $q \ge 3$.

\vs

Note that $\frac{q^2-\al q+\al^2}{q^2} = 1-\frac{\al}{q}+(\frac{\al}{q})^2$ is decreasing in $\frac{\al}{q}$ if $\frac{\al}{q} < \frac{1}{2}$. And for $q \ge 3$, $\frac{\al}{q}, \frac{\{\frac{q+\al}{k}\}}{\lfloor \frac{q+\al}{k}\rfloor} < \frac{1}{2}$. Therefore, to show that $$\frac{\lfloor \frac{q+\al}{k} \rfloor^2-\{\frac{q+\al}{k}\}\lfloor \frac{q+\al}{k}\rfloor+\{\frac{q+\al}{k}\}^2}{\lfloor \frac{q+\al}{k} \rfloor^2} \le \frac{q^2-\al q +\al^2}{q^2},$$ it suffices to show $$\frac{\{\frac{q+\al}{k}\}}{\lfloor \frac{q+\al}{k} \rfloor} \ge \frac{\al}{q}.$$ But $q\{\frac{q+\al}{k}\} = q(\frac{q+\al}{k}-\lfloor \frac{q+\al}{k}\rfloor)$, so the inequality reduces to $\frac{q}{k} \ge \lfloor \frac{q+\al}{k} \rfloor$, which is true since $\lfloor \frac{q+\al}{k}\rfloor = \lfloor \frac{q}{k} \rfloor$, since if $\frac{q}{k} < m \in \N$, then $\frac{q}{k} \le m-\frac{1}{k}$.

Next, observe that $$\frac{q^2+2\al q+4\al^2-6\al+3}{(q+1)^2} = \frac{(q+1)^2-(2-2\al)(q+1)+(2-2\al)^2}{(q+1)^2},$$ so since $\frac{2-2\al}{q+1} \le \frac{1}{2}$ for $q \ge 3$, as before it suffices to show that $$\frac{2-2\{\frac{q+\al}{k}\}}{\lfloor \frac{q+\al}{k}\rfloor +1} \ge \frac{2-2\al}{q+1}.$$ However, substituting $\{\frac{q+\al}{k}\} = \frac{q+\al}{k}-\lfloor \frac{q+\al}{k} \rfloor$, collecting terms with $q+\al$, and simplifying yields the equivalent $$\lfloor \frac{q+\al}{k} \rfloor + 1 \ge \frac{q+1}{k}.$$ And this is clearly true.
\end{proof}

\vsss

\section{Verifying the Gan-Loh-Sudakov Conjecture for Cayley Graphs}

We verify that our bound implies the bound in the Gan-Loh-Sudakov conjecture when $q \ge 7$. Take a finite Abelian group $G$ and a symmetric subset $S \sub G$ not containing $0$. Let $n=|G|$, $S_0 = S\cup\{0\}$, $d= |S|$, $q = \lfloor \frac{n}{|S_0|} \rfloor$, and $\al = \frac{n}{|S_0|}-q$. The benefit of working with $S_0$ is that the graph-theoretic bound takes the simpler form $$|T_{conj}| \le q{d+1 \choose 3}+{r \choose 3} = q{|S_0| \choose 3}+{\al|S_0| \choose 3}.$$ Note $$\Probb[S_0] = \frac{1}{|S_0|^2}\sum_{x,y \in S_0} 1_{S_0}(x+y) = \frac{1}{|S_0|^2}\left[\sum_{x,y \in S} 1_{S_0}(x+y) + 2\sum_{y \in S} 1_{S_0}(y)+1_{S_0}(0+0)\right].$$ Taking into account that for each $x \in S$ there is exactly one $y \in S$ for which $x+y = 0$, we see $$\Probb[S] = \frac{|S_0|^2}{|S|^2}\left[\Probb[S_0]-\frac{3|S|+1}{|S_0|^2}\right].$$ The number of triangles in our Cayley graph is thus $$\frac{1}{6}n|S|^2\Probb[S] = \frac{1}{6}(q+\al)|S_0|^3\left[\Probb[S_0]-\frac{3|S|+1}{|S_0|^2}\right].$$ For ease, let $M = \max\left(\frac{q^2-\al q+\al^2}{q^2},\frac{q^2+2\al q+4\al^2-6\al+3}{(q+1)^2},\gamma_0\right)$ so that, by Theorem 2 applied to $S_0$ (which is symmetric), we may bound the number of triangles by $$\frac{1}{6}(q+\al)M |S_0|^3 - \frac{1}{6}(q+\al)|S_0|(3|S_0|-2).$$ As one may check, this is less than $q{|S_0| \choose 3}+{\al |S_0| \choose 3}$ iff $$[(q+\al^3)-(q+\al)M]|S_0|^3+[3\al-3\al^2]|S_0|^2 \ge 0.$$ Therefore, it suffices to have $M \le \frac{q+\al^3}{q+\al}$. We have $\gamma_0 \le \frac{q+\al^3}{q+\al}$ for all $q \ge 7$ and any $\al \in [0,1]$. And, for any $q \ge 1, \al \in [0,1]$, $$\frac{q+\al^3}{q+\al} - \frac{q^2-\al q+\al^2}{q^2} = \frac{\al^3(q^2-1)}{q^2(q+\al)},$$ $$\frac{q+\al^3}{q+\al}-\frac{q^2+2\al q+4\al^2-6\al+3}{(q+1)^2} = \frac{(1-\al)^2(q-1)((2+\al)q+3\al)}{(q+1)^2(q+\al)}$$ are non-negative.

\vs

\section{Base Case $q=1$}

We finish by proving Theorems 1 and 2 when $|S| = \frac{n}{1+\al}$ for some $\al \in [0,1]$. Note $$\sum_{y \in S}\sum_{x \in G} 1_S(x+y) = \sum_{y \in S} |S| = |S|^2.$$ So, $$\sum_{x,y \in S} 1_S(x+y) = |S|^2-\sum_{x \not \in S} \sum_{y \in S} 1_S(x+y) = |S|^2-\sum_{x \not \in S} |(-x+S)\cap S|.$$ By pigeonhole, $|(-x+S)\cap S| \ge 2|S|-n$, and thus, $$|S|^2\Probb[S] \le |S|^2-\sum_{x \not \in S} (2|S|-n) = |S|^2(1-\al+\al^2).$$ As $1-\al+\al^2 = \frac{q^2-\al q+\al^2}{q^2}$ for $q=1$, Theorem 2 is established. Replacing $S$ with $2S$ in the appropriate places establishes Theorem 1 as well.

\vs

\section{Acknowledgments}

I would like to thank Po-Shen Loh for telling me the graph theoretic conjecture. I would also like to thank Adam Sheffer and Cosmin Pohoata for helpful comments.

\vs

\end{document}